\documentclass[a4paper,12pt]{article}

\usepackage[utf8]{inputenc}
\usepackage{lmodern} %to prevent use of ugly bitmap font
\usepackage{xcolor}
\usepackage{amsfonts,amssymb,amsmath,amsthm}
\usepackage{fourier} %pour le symbole \danger
\usepackage{enumerate}
\usepackage{hyperref}

\theoremstyle{plain}
\newtheorem{theorem}{Theorem}
\newtheorem{lemma}[theorem]{Lemma}
\newtheorem{corollary}[theorem]{Corollary}
\newtheorem{proposition}[theorem]{Proposition}

\theoremstyle{definition}
\newtheorem{definition}{Definition}

\theoremstyle{remark}
\newtheorem{remark}{Remark}

\newcommand{\vertiii}[1]{{\left\vert\kern-0.25ex\left\vert\kern-0.25ex\left\vert #1 
        \right\vert\kern-0.25ex\right\vert\kern-0.25ex\right\vert}}
\begin{document}

\title{Quasi-compactness for dominated kernels with application to quasi-stationary distribution theory}

\author{D. Villemonais$^{1,2}$}
\footnotetext[1]{IRMA, Université de Strasbourg, France}
\footnotetext[2]{Institut Universitaire de France\\ \texttt{denis.villemonais@unistra.fr}}

\maketitle

\begin{abstract}
	We establish a domination principle for positive operators, which provides an upper bound on the essential spectral radius and yields quasi‑com\-pactness criteria on weighted supremum spaces with Lyapunov type functions and local domination.
In particular, for kernels acting on such spaces,  we obtain $r_{ess}(P)\leq r_{ess}(Q)$
whenever 
$0\leq P\leq Q$ as kernels, a property that is known to fail in general on $L^p$ spaces, $p<+\infty$.

We then describe the asymptotics of iterates of positive quasi‑compact kernels, showing convergence, after suitable renormalization, towards a finite decomposition over eigenelements, and we study the long‑time behaviour of quasi‑compact continuous‑time semigroups. For the latter, we prove that measurability in time and quasi‑compactness at a single positive time imply quasi‑compactness at all times, exclude periodic behaviour, and entail convergence to eigenelements as time goes to infinity.

Finally, we apply these results to absorbed Markov processes and qua\-si‑stationary distributions. In this setting, the domination and Lyapunov criteria allow one to work in reducible situations and to relax classical regularity assumptions, for instance replacing strong Feller conditions by domination from a regular kernel or by locally uniformly integrable densities on suitable weighted supremum spaces.
\end{abstract}



\section{Introduction}	
\label{sec:intro}

Positive operators on Banach lattices form a unifying framework for Markov kernels, transfer operators, and many integral operators arising in analysis and dynamics.  Classical domination results~\cite{DoddsFremlin1979,Pitt1979,AliprantisBurkinshaw2006} show that compactness propagates under order domination, but the analogous question for the essential spectral radius (which measures the  `non‑compact' part of the operator) is delicate: on general Banach lattices, domination does not control the essential spectral radius (as shown by existing counter-examples in $L^2$ spaces). While some partial results exist for specific spaces, there is, to date, no simple and flexible domination principle for quasi‑compactness in the natural $L^\infty$ setting of sub Markov processes and positive kernels.

The first goal of this paper is to provide such a principle: we establish a new criteria for quasi-compactness based on domination conditions that are both easy to verify and particularly suited to the $L^\infty$ setting. Building on this, we develop Lyapunov‑type criteria for quasi‑compactness that only require local domination by compact kernels, and we derive general asymptotic decompositions for quasi‑compact kernels and continuous‑time semigroups, including propagation of quasi‑compactness from a single time, exclusion of periodic behaviour, and convergence to eigenelements.

The second goal is to show that this operator‑theoretic framework yields a unified and flexible theory for the study of quasi‑stationary distributions for absorbed Markov processes. Working on weighted 
$L^\infty$ spaces, we recover and extend existing quasi‑compactness type results based on strong Feller regularity and global Lyapunov conditions, replacing them by local domination hypotheses and locally uniformly integrable densities. In particular, we obtain existence and convergence to quasi‑stationary distributions for sub‑Markov processes on general state spaces, under assumptions that contains, as special cases, the strongly Feller Lyapunov frameworks developed previously for elliptic and hypoelliptic diffusions (including settings treated in~\cite{FerreRoussetEtAl2018,GuillinNectouxEtAl2024,CastroLambEtAl2021,BenaimEtAl2022}). 

\medskip

The study of compactness/quasi-compactness of a non-negative operator on a Banach lattice through domination properties has been initiated by the remarkable results obtained in \cite[Theorem~4.5]{DoddsFremlin1979}:
 it is proved therein that if a Banach lattice $B$ is norm continuous and if $P,Q$ are two  operators on $B$ with $0\leq P\leq Q$ and $Q$ compact, then $P$ is itself a compact operator (we also refer the interested reader to the contemporary article~\cite{Pitt1979}, where the author proves by simpler means an important special case). The extension of this domination condition for quasi-compactness has been studied since then, through the following problem: in general, if $0\leq P\leq Q$, is it true that the essential spectral radius of $P$ is smaller than the essential spectral radius of $Q$?
This is false in general: in~\cite{PagterShep1988}, the authors provide a counterexample in a  $L^2$ space. Nonetheless, quasi-compactness of dominated non-negative operators has been obtained under particular conditions on $B$ or $Q$. More precisely,   in~\cite{Caselles1987}, the author shows that if $B$ is a pre-dual function space and $Q$ is quasi-compact with essential spectral radius $0$, then this is also the case for $P$; in~\cite{Wu2004}, the author shows that if $Q$ satisfies a suitable regularity property and $B$ is the set of bounded measurable functions, then the essential spectral radius of $P$ is upper bounded by that of $Q$; several properties entailed by similar domination assumptions are also provided in~\cite{Martinez1993,Troitsky2004}.

Our first main result is that the quasi-compactness of $P$ can be derived from domination properties by the sum of two positive operators. Assuming that $0\leq P\leq K+S$, where $K$ is a non-negative compact operator and $S$ a non-negative operator, we show that $r_{ess}(P)\leq r(S)$, where $r_{ess}(P)$ denotes the essential spectral radius of $P$, as soon as the positive part $(P-S)_+$ of $P-S$ is well defined.
 Motivated by applications to probability theory, we focus on the setting where~$B$ is a weighted function space over a measurable space $(E,\mathcal E)$ which is countably generated, and  where $P$ and $Q$ are kernels on~$E$.  In this case, we show that  $0\leq P\leq Q$ does imply that $r_{ess}(P)\leq r_{ess}(Q)$. From this general result, we derive Lyapunov type criteria for quasi-compactness, with applications to positive kernels with densities and to positive kernels locally dominated by compact kernels. In particular, this allows us to extend the  results of~\cite{FerreRoussetEtAl2018,BenaimEtAl2022} from strong Feller kernel to kernels dominated by strong Feller kernels.

Our second main result concerns the analysis of the asymptotic behaviour of the iterates of quasi-compact kernels in an infinite norm setting. 
We show that the quasi-compactness of  a  kernel entails that its iterates converges, after proper renormalization, to a decomposition over eigenelements, extending results from~\cite{YosidaKakutani1941,Revuz2005,Hennion2007} to the non-conservative setting.
We also consider  the limiting behaviour of quasi-compact continuous time semigroups: we prove that measurability (in time) of the semigroup and quasi-compactness at one positive time is enough to deduce the quasi-compactness of the semigroup at all times $t$, the aperiodicity of the process, and convergence to eigenelements as $t\in[0,+\infty)$ tends to infinity.

We conclude the paper by applying the previous results to the study of existence and convergence to quasi-stationary distributions for absorbed Markov processes, which amounts to describe the convergence of the law of sub-Markov processes conditioned not to be absorbed when observed (see e.g.~\cite{ColletEtAl2012,Meleard2012}).
Among the several approaches that have been devised to study quasi-stationary distributions, the study of the spectral properties has been particularly successful for diffusion processes~\cite{LladserSanMartin2000,SteinsaltzEvans2007,CattiauxColletEtAl2009,CattiauxMeleard2010,LittinC.2012,KolbSteinsaltz2012,HeningKolb2019,HeningEtAl2021} and for birth and death processes~\cite{Doorn1991}, often through arguments relying on the self-adjoint properties of the infinitesimal generator or of an auxiliary generator (see also~\cite{ChampagnatVillemonais2017a} for a probabilistic approach to the study of these models and~\cite{SanchezEtAl2023} with in-depth considerations and practical criteria for Krein-Rutman type results). Recently, some authors focused specifically on the quasi-compactness of the transition kernel operator, using regularity assumptions such as strong Feller regularity, see e.g.~\cite{FerreRoussetEtAl2018,HinrichsKolbEtAl2018,CastroLambEtAl2021,BenaimEtAl2022,GuillinNectouxEtAl2024}. Using the quasi-compactness criteria and their consequences developed in this paper, we derive new flexible tools for the study of quasi-stationary distributions, which in particular extend the main results of~\cite{FerreRoussetEtAl2018,CastroLambEtAl2021,BenaimEtAl2022,GuillinNectouxEtAl2024}, replacing conditions on the (strong Feller) regularity of the process transition kernel by its domination by a regular kernel. We also derive criteria based on the existence of a locally uniformly integrable density.

\medskip We state and prove the quasi-compactness by domination criterion in Section~\ref{sec:QCdom} and provide Lyapunov type criteria in Section~\ref{sec:Lyap}.
 In Section~\ref{sec:QSDbyQC}, we focus on the asymptotic of iterates of positive quasi-compact kernels and on the long-time behaviour of quasi-compact continuous time semigroups. We conclude the paper in Section~\ref{sec:applis} with applications to the study of quasi-stationary distributions of sub-Markov processes.

\section{Quasi-compactness by domination arguments}
\label{sec:QCdom}

This section establishes a general domination principle controlling the essential spectral radius of positive operators on Banach lattices.

Let $(B,\|\cdot\|)$ be a Banach lattice (we refer the reader to~\cite[Chapter~4]{AliprantisBurkinshaw2006} for references and first properties of Banach lattices), with order $\geq$. A linear operator $Q$ on $B$ is called \textit{positive} if $Qf\geq 0$ for all  $f\in B$ such that $f\geq 0$. Given two linear operators $P$ and $Q$ on $B$, we say that $0\leq P\leq Q$ if $0\leq Pf\leq Qf$ for all element $f\in B$ such that $f\geq 0$. Given a bounded operator $Q$ on $B$, we denote by $\VERT Q\VERT$ the operator norm of $Q$ and by $r(Q)$ its spectral radius (in particular, positive linear operators on a Banach lattice are bounded, according to \cite[Theorem~12.3]{AliprantisBurkinshaw2006}).

The \emph{essential spectral radius} of an operator $Q$ on $B$ is  defined by
\[
r_{\mathrm{ess}}(Q) := \lim_{n \to \infty} \inf \left\{ \VERT Q^n - K \VERT^{1/n} : K \in \mathcal{K}(B) \right\},
\]
where \( \mathcal{K}(B) \) denotes the set of compact operators on \( B \) (we refer the reader  to~\cite[Chapter~5]{AliprantisBurkinshaw2006} for  definitions and properties of compact operators on Banach lattices). This definition captures the asymptotic  of the "non-compact part" of \( Q \) under iteration and provides a quantitative measure of how closely the powers of \( Q \) can be approximated by compact (or finite rank) operators. The essential spectral radius can also be defined as the common radius of the different essential spectra of $Q$, as detailed in~\cite[Section~1.4]{EdmundsEvans2018}.

An operator $Q$ on $B$ is said to be \textit{quasi-compact} if $r_{\mathrm{ess}}(Q)<r(Q)$, where $r(Q)$ is the spectral radius of $Q$. 

\begin{remark}
	Following another standard definition, \( r_{\mathrm{ess}}(Q) \) can be defined as the infimum of those \( \rho \geq 0 \) such that \( B \) admits a decomposition
	\[
	B = F_\rho \oplus H_\rho,
	\]
	into closed, \( Q \)-invariant subspaces with \( \dim F_\rho < \infty \), all eigenvalues of \( Q_{\rvert F_\rho} \) having modulus \( \geq \rho \), and \( r(Q_{\rvert H_\rho}) < \rho \) (see e.g.~\cite{HennionHerve2001}). The bounded linear operator \( Q \) is thus quasi-compact if there exists a decomposition
	\[
	B = F \oplus H,
	\]
	into closed, \( Q \)-invariant subspaces, where \( \dim F < \infty \), all eigenvalues of \( Q_{\rvert F} \) have modulus equal to the spectral radius \( r(Q) \), and the restriction \( Q_{\rvert H} \) satisfies \( r(Q_{\rvert H}) < r(Q) \).
		The link between this alternative definition and the definition we give in terms of compact approximation is a by-product of the properties established in the first chapter of~\cite{Ruston2004}. 
\end{remark}

Finally,  given a bounded linear operator $Q$ on $B$, we say that  the \textit{positive part} of $Q$, denoted $Q_+$, is well defined if, for all $f\geq 0$,
\[
Q_+f:=\sup\{Qg,\ g\in B,\,0\leq g\leq f\}
\]
defines a positive linear operator on $B$. In particular, $Q_+$ is well defined if $B$ is Dedekind complete (see e.g.~\cite{AliprantisBurkinshaw2006} page 13), which is the case for $L^p(\mu)$ spaces, $1\leq p\leq \infty$, on a countably generated measured space $(E,\mathcal E,\mu)$ with $\mu$ $\sigma$-finite (see e.g. Example 2.52 p.48 in~\cite{BanasiakArlotti2006}).

\begin{theorem}
    \label{thm:QCbyDOM} 
    Let $P$ be a positive operator on $B$ and assume that there exists two positive operators $K$ and $S$ such that $K$ is compact, $(P-S)_+$ is well defined and 
    $0\leq P\leq K+S$. Then $r_{ess}(P)\leq r(S)$.
\end{theorem}

Theorem~\ref{thm:QCbyDOM} immediately entails the following corollary.

\begin{corollary}
	If, in addition to the assumptions of Theorem~\ref{thm:QCbyDOM}, we have $r(S)<r(P)$, then $P$ is quasi-compact.
\end{corollary}

\begin{proof}[Proof of Theorem~\ref{thm:QCbyDOM}]
	We deduce from $0\leq P\leq S+K$ that
	\begin{align*}
		0\leq (P-S)_+ \leq K,
	\end{align*}
	where $(P-S)_+$ is a positive kernel defined by
	\begin{align*}
		\delta_x (P-S)_+ = (\delta_x P-\delta_x S)_+,\ \forall x\in E.
	\end{align*}
	We also have
	\begin{align*}
		P=(P-S)_+ +S\wedge P,
	\end{align*}
	where $0\leq S\wedge P\leq S$, so that $\|S\wedge P\|\leq \|S\|$.

	We deduce that, for all $n\geq 0$,
	\begin{align*}
		P^n = S^n +\sum_{i=1}^n S^{i-1} (P-S)_+ S^{n-i} + \sum_{1\leq i<j\leq n} S^{i-1} (P-S)_+ S^{j-i} (P-S)_+ S^{n-i-j}+ K_n,
	\end{align*}
	where $K_n$ is a finite sum of products of operators where $(P-S)_+$ appears at least three times. According to the following lemma, we deduce that $K_n$ is compact.
	
	\begin{lemma}
		\label{lem:compactproduct}
		Let $S_1,S_2,S_3,S_4,Q_1,Q_2,Q_3$ be non-negative bounded operators and $K_1,K_2,K_3$ be non-negative compact operators such that $0\leq Q_i\leq K_i$ for all $i\in\{1,2,3\}$. Then the operator
		\begin{align*}
			S_1 Q_1 S_2 Q_2 S_3 Q_3 S_4
		\end{align*}
		is compact.
	\end{lemma}
	
	\begin{proof}[Proof of Lemma~\ref{lem:compactproduct}]
		We observe that $0\leq S_1Q_1S_2\leq S_1K_1 S_2$, $0\leq Q_2S_3\leq K_2 S_3$ and $0\leq Q_3S_4\leq K_3S_4 $, where the operators $S_1K_1 S_2$, $K_2 S_3$ and $K_3S_4$ are compact. Hence, according to \cite[Theorem 5.14]{AliprantisBurkinshaw2006}, the product of the three operators is a compact operator.
	\end{proof}
	
	Considering the first three terms in the above decomposition, we get that
	\begin{align*}
		\left\| P^n-K_n\right\|\leq \|S\|^n+n \|(P-S)_+\|\,\|S\|^{n-1}+\frac{n(n-1)}{2}  \|(P-S)_+\|^2\,\|S\|^{n-2}.
	\end{align*}
	We deduce that
	\begin{align*}
		r_{ess}(P)&=\left(r_{ess}(P^n)\right)^{1/n}\leq \|S\|^{1-2/n} \left(\|S\|^2+ \|(P-S)_+\|\,\|S\|+ \|(P-S)_+\|^2\right)^{1/n}\\
		&\xrightarrow[n\to+\infty]{} \|S\|.
	\end{align*}
	We thus proved that $r_{ess}(P) \leq \|S\|$.

	We conclude the proof by proving that this inequality implies that $r_{ess}(P)\leq r(S)$. Indeed, we get from $0\leq P\leq S+K$ that, for all $n\geq 1$,
	\begin{align*}
		0\leq P^n \leq S^n + K',
	\end{align*}
	with $K'$ given by the sum of products of bounded operators where the compact operator $K$ appears at least once. This entails that $K'$ is a compact operator and hence, according to the first part of the proof, that $r_{ess}(P^n)\leq \|S^n\|$ and hence that 
	\begin{align*}
		r_{ess}(P)=r_{ess}(P^n)^{1/n}\leq \|S^n\|^{1/n}\xrightarrow[n\to+\infty]{} r(S),
	\end{align*}
	where we used  Gelfand formula for the last inequality. This concludes the proof of Theorem~\ref{thm:QCbyDOM}.
\end{proof}

We specialize now our results to the case of weighted spaces of measurable functions. 
More precisely, considering  a  countably generated measurable space $(E,\mathcal E)$  and a measurable function $V:E\to (0,+\infty)$, we define the function space
\[
L^\infty(V):=\left\{f:E\to\mathbb C,\,\exists C>0,\text{ such that }|f(x)|\leq C\,V(x),\ \forall x\in E\right\},
\]
endowed with the  weighted supremum norm $\|f\|_V=\sup_{x\in E}\frac{|f(x)|}{V(x)}$. We say that a functional $R:E\times\mathcal E\to\mathbb R$ is a  kernel if, for all $x\in E$, $R(x,\cdot)$ is a $\sigma$-finite measure, and, for all $A\in \mathcal E$, $R(\cdot,A)$ is measurable. A kernel such that $\|R(\cdot,V)\|_V<\infty$ is called a kernel from $E$ to $L^\infty(V)$ and it defines a linear operator on $L^\infty(V)$. In the following result, $r_{ess}(R)$ is then the essential spectral radius of the induced linear operator.

 We obtain the following corollary to Theorem~\ref{thm:QCbyDOM}, which extends~\cite[Corollary 3.7]{Wu2004}, by removing the regularity condition on $P$.

\begin{corollary}
	\label{cor:BbE}
	Let $P$ and $Q$ be kernels from $E$ to $L^\infty(V)$ such that $0\leq P\leq Q$. Then $r_{ess}(P)\leq r_{ess}(Q)$.  
\end{corollary}

\begin{proof}[Proof of Corollary~\ref{cor:BbE}]
	The proof relies on Theorem~\ref{thm:QCbyDOM} and on the following equivalent definition of the positive part of a kernel: given a kernel $R$ from $E$ to $L^\infty(V)$, we have
	\begin{align}
		R_+f(x):=(\delta_xR)_+f,
	\end{align}
	where $(\delta_x R)_+$ is the positive part of the measure $\delta_x R$. The fact that  $R_+$ is itself a kernel is a consequence of Section~2 in~\cite{DubinsFreedman1964}. In addition, we have $\VERT R_+\VERT\leq \VERT R\VERT$, since, for all non-negative $f\in L^\infty(V)$ and all $x\in E$, 
	\begin{align*}
		|R_+ f(x)|\leq R_+ |f|(x) = R|f\mathbf{1}_{A_x}|(x)\leq \VERT R\VERT \,\|f\mathbf{1}_{A_x}\|_V\leq  \VERT R\VERT \,\|f\|_V,
	\end{align*}
	where $A_{x}$ is the positive set in the Hahn decomposition of the measure $\delta_x R$. Similarly, $r(R_+)\leq r(R)$.

	By definition of $r_{ess}(Q)$, for any $\varepsilon$, we can find a finite rank operator $K$ and a (possibly signed) bounded operator $S$ such that $Q=K+S$ and $\VERT S^n\VERT^{1/n} \leq r_{ess}(Q)+\varepsilon$ for some $n\geq 1$. Since we also have $0\leq P^n\leq Q^n =K'+S^n$, with $K'$ a finite rank operator, we can assume without loss of generality that $n=1$ (replacing $Q$ by $Q^n$).
		
	 Since $K$ is a finite rank operator, its image is a finite subspace of $L^\infty(V)$ and hence there exist $h_1,\ldots,h_d$ bounded continuous functionals from $L^\infty(V)$ to $\mathbb R$, and $a_1,\ldots,a_d\in L^\infty(V)$ such that, for all $x\in E$ and all non-negative $f\in L^\infty(V)$, 
	\[
	Kf(x)=\sum_{i=1}^d h_i(f) a_i(x)\leq \bar Kf(x):=\sum_{i=1}^d h_{i,+}(f) |a_i|(x)
	\]
	where $h_{i,+}(f)=\sup\{h_i(g),\,0\leq g\leq f\}$. We deduce that $0\leq P\leq \bar K+S_+$, where $\bar K$ is a non-negative finite rank operator and $\VERT S_+\VERT\leq \VERT S\VERT \leq r_{ess}(Q)+\varepsilon$.
	
	Using Theorem~\ref{thm:QCbyDOM}, we deduce that $r_{ess}(P)\leq r_{ess}(Q)+\varepsilon$, and hence that $r_{ess}(P)\leq r_{ess}(Q)$.
\end{proof}

\begin{remark}
	In $L^p$ spaces, Corollary~\ref{cor:BbE} does not hold (see the counterexample in~\cite{PagterShep1988}), precisely because $\VERT S_+\VERT \leq\VERT S\VERT$ does not hold  in general. We refer the reader to~\cite{AbramovichWickstead1995} for several examples of disaccording compactness properties of $S$ and $S_+$ on $L^p$ spaces.
\end{remark}

\section{Lyapunov type criterion for quasi-compactness in weighted spaces}

\label{sec:Lyap}

In this section, we translate the abstract domination principle into concrete quasi‑compactness criteria for kernels on weighted supremum spaces. By combining a global Lyapunov drift outside a region with local domination by compact kernels inside the region, we obtain simple conditions ensuring quasi‑com\-pactness and explicit bounds on the essential spectral radius. These criteria both generalize strong Feller‑based results and are particularly well suited to non‑Feller kernels and non‑topological state spaces.

Let $E$ be a measurable space endowed with a countably generated $\sigma$-field $\mathcal E$ and $V:E\rightarrow (0,+\infty)$ be a measurable function. Let $P$ be a kernel on $E$ defining a bounded operator on $L^\infty(V)$.

We consider the following hypothesis.

\medskip\noindent\textbf{Assumption (H).} There exist   a 
measurable $E_K\subset E$ such that:
\begin{itemize}
	\item[(H1)] \textit{(Global Lyapunov criterion).} There exists $\theta_1\in(0,r(P))$ such that
	\begin{align*}
		PV(x)\leq \theta_1 V(x),\ \forall x\in E\setminus E_K.
	\end{align*}
	\item[(H2)] \textit{(Local domination).} There exists $\theta_2\in(0,r(P))$ and a positive compact operator $K$ on $L^\infty(V)$ such that
	\begin{align*}
		0\leq Pf(x)\leq Kf(x)+\theta_2 V(x)\|f\|_V,\ \forall x\in E_K,\  f\in L^\infty(V),\ f\geq 0. 
	\end{align*}
\end{itemize}

Under the above assumptions, we have the following result. 

\begin{proposition}
	\label{prop:H1H3}
	If Assumption~(H) holds, then the operator $P$ is quasi-compact in $L^\infty(V)$, with $r_{ess}(P)\leq \theta_1\vee \theta_2$.
\end{proposition}

\begin{remark}
	If the above Assumptions (H1) and (H2) hold at greater times, that is if there exists $n_0\geq 1$ and $\theta_1,\theta_2\in(0,r(P^{n_0}))$  such that $	P^{n_0}V(z)\leq \theta_1 V(z),\ \forall z\in E\setminus E_K$ and $0\leq P^{n_0}f(x)\leq Kf(x)+\theta_2 V(x)\|f\|_V,\ \forall x\in E_K,\  f\in L^\infty(V),\ f\geq 0$, then $P^{n_0}$ is quasi-compact according to Proposition~\ref{prop:H1H3} and hence $P$ is quasi-compact.
\end{remark}

\begin{remark}
	\label{rem:phi}
	For practical use, we note that $r(P)$ can be bounded from below using Lyapunov type criterion. Following the idea of~\cite{ChampagnatVillemonais2017a}, let us assume that there exists a non-zero non-negative function $\varphi\in L^\infty(V)$ such that $P\varphi\geq \theta \varphi$ for some $\theta>0$. Then $r(P)\geq \theta$. Indeed, this implies that $P^n\varphi\geq \theta^n \varphi$, hence 
	\begin{align*}
		r(P)&=\lim_{n\to+\infty} \VERT P^n\VERT^{1/n}
		\geq \lim_{n\to+\infty} \frac{\|P^n \varphi\|_V^{1/n}}{\|\varphi\|_V^{1/n}}
		\geq  \theta.
	\end{align*}
\end{remark}

\begin{proof}[Proof of Proposition~\ref{prop:H1H3}]
	Let $S$ be the positive kernel defined by
	\[
	Sf(x)=\mathbf 1_{x\notin E_K} Pf(x)+\mathbf 1_{x\in E_K} (P-K)_+f(x),\ \forall x\in E,\  f:L^\infty(V).
	\]
	According to  (H1) and (H2), we have $\VERT S\VERT\leq \theta_1\vee\theta_2<r(P)$.
	
	Then we have $0\leq P\leq Q:=K+S$ with $K$ compact and
	\[
	r_{ess}(Q)\leq r(S)\leq \VERT S\VERT\ < r(P).
	\]
	This and Theorem~\ref{thm:QCbyDOM} conclude the proof of Proposition~\ref{prop:H1H3}.
\end{proof}

In practice, there are many ways to prove that an operator $K$ on $L^\infty(V)$ is compact. However, it is often simpler when the base set on which the functions are defined is localized. To deal with this situation, we propose two corollaries of Proposition~\ref{prop:H1H3} which make use of localized versions of (H2). The first one is particularly useful when (H1) holds for arbitrarily small values of $\theta_1$.

\begin{corollary}
	\label{cor:loc1}
	Assume that (H1) holds and that, for some real $\theta_3>0$, integer $k\geq 1$, and  some positive operator $G$ on $L^\infty(V)$ with $G^k$ compact,
	\begin{align}
		\label{eq:cor:loc1}
		0\leq P(\mathbf 1_{E_K}f)(x)\leq G f(x)+\theta_3V(x)\|f\|_V,\quad\forall x\in E_K\text{ and }f\in L^\infty(V).
	\end{align}
	If 
	\[
	\theta:=(\theta_3+\theta_1)\sum_{i=0}^{k-1} \VERT G^{i}\VERT \, \VERT P^{k-i}\VERT <r(P)^{k+1},
	\]
	then  Proposition~9 applies to $P^{k+1}$. In particular, $P$ is quasi-compact and we have $r_{ess}(P)\leq \theta_1\vee \theta^{1/(k+1)}$.
\end{corollary}

\begin{remark}
	\label{rem:strongFeller}
	Assume that $E$ is a separable topological space, $E_K$ is a relatively compact set with $\sup_{E_K} V<\infty$, and that 
	\[
	P(\mathbf 1_{E_K}f)(x)\leq G(\mathbf 1_{E_K}f)(x)\text{ for all $x\in E_K$ and all $f\in L^\infty(V)$},
	\]
	for some strong Feller kernel $G$  on $E_K$. Since $G$ is strong Feller, then $G^2$ is ultra-Feller and hence compact on $L^\infty(V_{\vert E_K})$. In particular, its extension to  $L^\infty(V)$, defined by $f\in L^\infty(V)\to \mathbf 1_{E_K}G(\mathbf 1_{E_K}f)$, is compact and dominates $P$. This implies that~\eqref{eq:cor:loc1} holds with $\theta_3=0$ and $k=2$. If in addition (H1) holds  with $\theta_1(\VERT P^{2}\VERT + \VERT G\VERT \, \VERT P\VERT )<r(P)^{3}$, then the assumption of Corollary~\ref{cor:loc1} hold.
	
	This shows that Corollary~\ref{cor:loc1} generalizes the main result of~\cite{FerreRoussetEtAl2018} to non-regular kernels. More precisely, the assumptions therein entail that, for all $n>0$, there exists a compact subset $E_n$ of $E$ such that (H1) holds with $\theta_1=1/n$ and that $P$ restricted to $E_n$ is strong Feller. Our result entails that the conclusion of~\cite{FerreRoussetEtAl2018} still holds under a weaker assumption, namely that $P$ is dominated by (not necessarily equal to) an operator which is strong Feller on each~$E_n$.
\end{remark}

\begin{proof}[Proof of Corollary~\ref{cor:loc1}]
	We have, for all $x\in E_K$ and $f\in L^\infty(V)$,  using the domination assumption~\eqref{eq:cor:loc1} and Assumption~(H1),
	\begin{align*}
		P^{k+1}f(x)&=P(\mathbf 1_{E_K}P^{k}f)(x)+P(\mathbf 1_{E_K^c}PP^{k-1}f)(x)\\
		&\leq GP^{k}f(x)+\theta_3\|P^{k}f\|_V V(x)+\theta_1\|P^{k}f\|_V V(x)\\
		&\leq G^kPf(x)+(\theta_3+\theta_1)\sum_{i=0}^{k-1} \VERT G^{i}\VERT \, \|P^{k-i}f\|_V  V(x),
	\end{align*}
	by immediate iteration. In particular, we deduce that Assumption~(H2) with $\theta_2=\theta$ holds for $P^{k+1}$.
	
	In addition, for all $x\notin E_K$,
	\begin{align*}
		P^{k+1}V(x)\leq \|P^k\|_V PV(x)\leq \VERT P\VERT^k\,\theta_1 V(x),
	\end{align*}
	where $\VERT P\VERT^k\,\theta_1 \leq \theta <r(P)^{k+1}=r(P^{k+1})$ under our assumptions, so that Assumption~(H1) holds  for $P^{k+1}$. This and Proposition~\ref{prop:H1H3} conclude the proof of Corollary~\ref{cor:loc1}.
\end{proof}

Our second localized set of assumption makes use of moments properties and, compared to the previous corollary,  is particularly useful when assumption (H1) is not known for arbitrarily small values of $\theta_1$.

\begin{corollary}
	\label{cor:localdomqsdbis}
Assume that there exists $\theta_4\in[0,r(P))$ and $A>0$ such that
	\[
\theta_2:=\sup_{x\in E_K} \frac{1}{V(x)}P(V\mathbf 1_{V>A})(x)+\theta_4<r(P)
	\]
	and such that there exists a compact operator $K_A$ from the function space $L^\infty(V_{\vert \{V\leq A\}})$ to $L^\infty(V_{\vert E_K})$ such that 
	\begin{align}
		\label{eq:cor:local2}
		0\leq P(f\mathbf 1_{V\leq A})(x)\leq K_A f(x)+\theta_4V(x)\|f\|_V,\quad\forall x\in E_K\text{ and }f\in L^\infty_+(V_{\vert \{V\leq A\}}).
	\end{align}
	Then (H2) holds. In particular,  if in addition (H1) holds, then $P$ is quasi-compact on $L^\infty(V)$ with $r_{ess}(P)\leq \theta_1\vee \theta_2$.
\end{corollary}

\begin{remark}
	In~\cite{BenaimEtAl2022}, the authors assume, instead of~\eqref{eq:cor:local2}, that $E_K$ is relatively compact and $P$ strongly Feller, and deduce that $P$ is quasi-compact. Similarly as in Remark~\ref{rem:strongFeller}, our result shows that the conclusion of~\cite{BenaimEtAl2022} holds under the much weaker constraint that $P$ is dominated by (not necessarily equal to) a strong Feller semigroup on a sub-level set of $V$. See Proposition~\ref{prop:strongFeller} for details in the context of sub-Markov continuous time processes.
\end{remark}

\begin{proof}[Proof of Corollary~\ref{cor:localdomqsdbis}]
	We have, for all $x\in E_K$ and all $f\in L^\infty(V)$, $f\geq 0$,
	\begin{align*}
		0\leq Pf(x)&\leq K_A(f\mathbf 1_{V\leq A})(x)+ P(f\mathbf 1_{\{V>A\}})(x)+\theta_4 V(x)\|f\|_V\\
		&=Kf(x)+ P(f\mathbf 1_{\{V>A\}})(x)+\theta_4 V(x)\|f\|_V\\
		&\leq Kf(x)+\theta_2V(x)\|f\|_V,
	\end{align*}
	where
	$Kf:=\mathbf 1_{V\leq E_K}K_A(f_{\vert\{V\leq A\}})$ defines a compact operator on $L^\infty(V)$, indeed: by assumption, $f\mapsto K_A(f_{\vert\{V\leq A\}})$ maps the unit ball of $L^\infty(V)$ to a compact subset of $L^\infty(V_{\vert E_K})$, which is continuously embedded in $L^\infty(V)$ through $f\mapsto \mathbf 1_{E_K}f$.
	This shows that Assumption~(H2) holds.
\end{proof}

To conclude this  section, we provide a simple criterion for quasi-compactness, which leverages on the two  localized criteria stated in Corollaries~\ref{cor:loc1} and~\ref{cor:localdomqsdbis}: we show that quasi-compactness can be deduced from the existence of locally uniformly integrable densities, without regularity assumption nor assuming that $E$ is a topological space.

\begin{proposition}
	\label{prop:density}
	Assume that Assumption~(H1) holds and that there exists a measure $\nu$ on $E$ and a measurable $p:E\times E\to [0,+\infty)$ such that, for all $x\in E_K$,
	\[
	Pf(x)=\int_E f(y) p(x,y)\,\nu(\mathrm dy).
	\]
	Assume in addition that one of the two following conditions holds:
	\begin{enumerate}[i)]
		\item we have $\theta_1 \VERT P\VERT <r(P)$,  $\nu(V\mathbf 1_{E_K})<\infty$ and 
		\[
		\sup_{x\in E_K} \frac1{V(x)}\int_{E_K} p(x,y)\mathbf 1_{p(x,y)>B}V(y)\,\nu(\mathrm dy)\xrightarrow[B\to+\infty]{} 0;
		\]
		\item for all $A>0$, $\nu(V\mathbf 1_{V\leq A})<\infty$ and
		\[
		\sup_{x\in E_K}  \frac1{V(x)}\int_{V\leq A} p(x,y)\mathbf 1_{p(x,y)>B}V(y)\,\nu(\mathrm dy)\xrightarrow[B\to+\infty]{} 0;
		\]
	\end{enumerate}
	Then $P$ is quasi-compact in $L^\infty(V)$.
\end{proposition}

\begin{proof}[Proof of Proposition~\ref{prop:density}]
	Let us  consider case {\em i)}. We check that Corollary~\ref{cor:loc1} applies with $k=1$ Let $B>0$ be large enough so that
	\[
	\theta_3:=\sup_{x\in E_K}  \frac1{V(x)} \int_{E_K} p(x,y)\mathbf 1_{p(x,y)>B}V(y)\,\nu(\mathrm dy)< \frac{r(P)-\theta_1 \VERT P\VERT }{\VERT P\VERT }.
	\]
	Then we have, for all $x\in E_K$ and all $f\in L^\infty_+(V)$,
	\begin{align*}
		0 \leq P(\mathbf 1_{E_K}f)(x)&= \int_{E_K} p(x,y)\mathbf 1_{p(x,y)\leq B}\,f(y)\,\nu(\mathrm dy)+\int_{E_K} p(x,y)\mathbf 1_{p(x,y)> B}\,f(y)\,\nu(\mathrm dy)\\
		&\leq B\, \int_{E_K} f(y)\,\nu(\mathrm dy)+\theta_3 \|f\|_V\,V(x).
	\end{align*}
	Setting $G:f\in L^\infty(V)\mapsto B\nu(f\mathbf 1_{E_K})\mathbf 1_{E_K}$ (which defines a compact operator) concludes the proof.
	
	The case {\em ii)} can be treated almost identically, using Corollary~\ref{cor:localdomqsdbis} and is thus not detailed.
\end{proof}

\section{Asymptotics of iterates of quasi-compact kernels and of quasi-compact semigroups}

\label{sec:QSDbyQC}

In this section, quasi‑compactness is used to describe the fine asymptotics of positive kernels and continuous‑time semigroups on weighted 
$L^\infty$-spaces. For discrete‑time kernels, we obtain a decomposition of iterates along finitely many eigenelements, with uniform control in the weighted norm and a precise description of the underlying cyclic structure of the state space. For measurable continuous‑time semigroups, we show that quasi‑compactness at a single time propagates to all times, rules out periodic behaviour, and yields convergence to eigenelements as time goes to infinity.

\subsection{Iterates of quasi-compact kernels}
\label{sec:QSDbyQCdis}

Let $E$ be a measurable space endowed with a countably generated $\sigma$-field $\mathcal E$ and $P$ a kernel on $E$.
We also assume that we are given a function $V:E\to(0,+\infty)$.

\begin{theorem}
	\label{thm:quasi-compact-and-QSD}
	Assume that
$P$ acts as a quasi-compact operator on $L^\infty(V)$.
	Then there exists an integer $d\geq 1$, an integer valued bounded function $j:E\to \mathbb N$, a finite set $I$ and some non-zero measures $\nu_{i}$ on $E$ such that $\nu_i(V)<+\infty$ and non-identically zero non-negative functions $\eta_{i}\in L^\infty(V)$ for each $i\in I$, such that, for all $f\in L^\infty(V)$, all $n\geq 1$, all $k\in\{0,1,\ldots,d-1\}$ and all $x\in
	E$,
\begin{equation}
	\label{eq:etathm}
	\left|\frac{r(P)^{-dn-k}}{(dn+k)^{j(x)}} P^{nd+k}f-\sum_{i\in I} \eta_{i,k}(x)\nu_{i,k}(f)\right|\leq \alpha_{nd+k} V(x)\|f\|_{V},
\end{equation}
	where $\nu_{i,k}=\nu_i P^k$, $\eta_{i,k}=r(P)^{-k}\eta_i$ and $\alpha_{n}$ goes to $0$ when $n\to+\infty$.
	
	In addition, there exist disjoint subsets $E_i\subset E$, $i\in I$, such that
	\begin{enumerate}
		\item 
		$F_{i}:=E_{i}\cup\{x\in E,\ \exists n\geq 0,\ P^{nd}\mathbf 1_{E_{i}}=0\}$ is a closed subset for 
		$P$, meaning that, for all $x\in F_i$, $P^d \mathbf 1_{F_i^c}=0$,
		\item $\nu_{i}$ is supported by $F_{i}$ and 
		$\nu_iP^d=r(P)^d\nu_i$,
		\item on each $E_{i}$, $j=0$, $\eta_i>0$ and $\eta_{i'}=0$ for all $i'\neq i$.
	\end{enumerate} 
\end{theorem}

\begin{remark}
	\label{rem:not-countablygenerated}
	In the proof of Theorem~\ref{thm:quasi-compact-and-QSD}, we make use of~\cite[Proposition~3.78, Chapter~6]{Revuz2005}, which requires $\mathcal E$ to be countably generated. Besides that, the proof requires only minor modifications  to get the same result without the countable generation assumption. 
\end{remark}

\begin{remark}
	We apply this result to sub-Markov kernel (meaning that $P\mathbf 1_E\leq  \mathbf 1_E$) in Section~\ref{sec:applis}. Note that,
	when $P$ is a Markov kernel (meaning that $P\mathbf 1_E= \mathbf 1_E$), 
	 the result is known~\cite{Hennion2007} and we have
	 $j\equiv 0$ (see also the extended preprint version~\cite{Hennion2007arxiv}), see also~\cite{Revuz2005} and~\cite{YosidaKakutani1941} for quasi-compact Markov kernels on $B_b(E)$.
\end{remark}

\begin{proof}[Proof of Theorem~\ref{thm:quasi-compact-and-QSD}]
	Replacing $P$ by the kernel by its $V$-transform
	\[
	f\in L^\infty(E)\to (x\mapsto \frac{1}{V(x)\,\VERT P\VERT }P(V f)(x)),
	\]
	we can assume without loss of generality  that $P$ is a sub-Markov kernel (meaning that $P\mathbf 1_E\leq  \mathbf 1_E$). Let $\partial\notin E$ and let $(X_n)_{n\in\mathbb N}$ be an associated absorbed Markov process on $E\cup\{\partial\}$ with absorption at $\partial$, which means that $X$ has the following probability of transitions $\mathbb P_x(X_1\in A)=P\mathbf 1_A(x)$ for all $A\subset E$ and $\mathbb P_x(X_1=\partial)=1-P\mathbf 1_E(x)$ for all $x\in E$, and $\mathbb P_\partial(X_1=\partial)=1$.

The proof relies on the decomposition of the state space $E$ into subsets $E_1, E'_1, \ldots,E_p,E'_p,E'_{p+1}$ such that the semigroup associated to $X$ is well behaved on each $E_k,E'_k$: each $E_k$ can be decomposed into communication classes where it converges (at times that are multiple of some period) to a one-dimensional operator, while each $E'_k$ is such that the process escapes exponentially fast (compared to $\rho(P)$) from $E'_k$. This will allow us to apply~\cite[Theorem~4.1]{ChampagnatVillemonais2024}. 

The proof  is divided in four steps. In the first step, we build  $E_1$ and $E'_1$. In the second step, we build by induction $E_1,E'_1,\ldots,E_p,E'_p$ and $E_{p+1}$. In the third step, we apply the above cited result to the process at times that are multiples of a given integer $d$. In the four step, we conclude by extending the convergence to any time.

\medskip\noindent\textit{Step 1. First sets in the state space decomposition.}	Since $P$ is a quasi-compact operator on $B=L^\infty(V)$, we deduce from~\cite{Sasser1964} that there exists a non-negative non-zero function $\eta_1\in B$ such that $P\eta_1=r(P)\eta_1$. Let $E_1:=\{x\in E,\ \eta_1(x)>0\}$ and define the Markov kernel $T$ on $E$ by
	\begin{align*}
		Tf(x)=\frac{1}{r(P)\eta_1} P(\eta_1 f)(x).
	\end{align*}
	The Markov kernel $T$ is the so called $\eta_1$ transform of $P$. We emphasize, for later use, that $E\setminus E_1$ is a closed subset for the Markov chain $X$ (meaning that $X$ cannot escape $E\setminus E_1$ before absorption).
	
	We set $V_1=\frac{V}{\eta_1}$ on $E_1$. One easily checks that $T$ is a quasi-compact operator on $L^\infty(V_1)$, with underlying state space $E_1$.  We deduce from~\cite[Corollary IV.3]{Hennion2007arxiv} and the adaptation of~\cite[Proposition~3.8]{Revuz2005} (this adaptation is suggested in~\cite{Hennion2007arxiv}, and one can also adapts the arguments of~\cite{YosidaKakutani1941}, the main addition being the non-trivial fact that the kernel $T$ is power-bounded) that there exist a finite number $r\geq 1$ of integers $d_\rho$, $\rho=1,\ldots,r$, and measurable functions $U_{\rho,\delta}\in L^\infty(V_1) $, $\delta=1,\ldots,d_\rho$ such that $\sum_{\rho}\sum_{\delta}U_{\rho,\delta}=1$ and probability measures $m_{\rho,\delta}$ carried by the pairwise disjoint sets $E_{\rho,\delta}=\{U_{\rho,\delta}=1\}$ such that, if $d$ denotes the least common multiple of the $d_\rho$, then for every $k$, 
	\begin{align}
			\label{eq:decomp}
		\left\VERT (T)^{nd+k}-\sum_{\rho=1}^r\sum_{r=1}^{d_\rho}U_{\rho,\delta-k}\otimes m_{\rho,\delta}\right\VERT\leq C\rho^n,
	\end{align}
	for some $\rho<1$ and $C>0$. In addition, the sets $E_\rho=\cup_\delta E_{\rho,\delta}$ are closed set, and the sets $E_{\rho,\delta}$ and the number $d_\rho$ are the corresponding cyclic classes and period.
	
	In terms of our original sub-Markov process, this entails that  the sets $E_{\rho,\delta}$ are absorbing for the Markov chain $(X_{nd})_{n\geq 0}$ and that, for all $f\in L^\infty(V)$ and all $x\in E_{\rho,\delta}$, we have
	\begin{align}
		\label{eq:criticalsink}
		\left| r(P)^{-nd}\mathbb E_x\left(f(X_{nd})\mathbf 1_{X_n\in E_{1,\rho,\delta}}\right)-\eta_{1,\rho,\delta}(x) \nu_{1,\rho,\delta}(f)\right|\leq C\rho^n V(x)\|f\|_V,
	\end{align}
	where $E_{1,\rho,\delta}:=E_{\rho,\delta}$, $\eta_{1,\rho,\delta}:=\eta_1 U_{\rho,\delta}$ is positive on $E_{1,\rho,\delta}$ and $\nu_{1,\rho,\delta}=m_{\rho,\delta}(./\eta_1)$.  Finally, setting $E'_1=E_1\setminus \cup_\rho E_\rho$, we have  $m_{\rho,\delta}(E'_1)=0$ and hence, for all $f\in L^\infty(V)$  and all $x\in E'_1$,
	\begin{align}
		\label{eq:sink}
		\left| r(P)^{-nd}\mathbb E_x\left(f(X_{nd})\mathbf 1_{X_n\in E'_1}\right)\right|\leq C\rho^n V(x)\|f\|_V.
	\end{align}
	This concludes the first step.
	
	\medskip\noindent \textit{Step~2. Inductive construction of the other subsets.} As mentioned above, the subset $E\setminus E_1$ is closed. We denote by $R=P_{\vert E\setminus E_1}$ the kernel $P$ restricted to $E\setminus E_1$. 
	
	If $r(R)<r(P)$, then we set $p=1$ and $E'_{p+1}=E\setminus E_1$. The definition of the spectral radius entails that~\eqref{eq:sink} holds with $E'_{p+1}$ instead of $E'_1$.
	
	Otherwise, we have $r(R)=r(P)$ and, since it is dominated by $P$, the kernel $R$ is quasi-compact according to Theorem~\ref{thm:QCbyDOM}. 
		As a consequence, Step~1 applies to $R$ and we deduce that there exists $E_2\subset E\setminus E_1$ such that~\eqref{eq:criticalsink} and~\eqref{eq:sink} holds, with the  subsets $E_{1,\rho,\delta}$ replaced by  the closed subsets $E_{2,\rho,\delta}\subset E_2$ (the range of $\rho$ and $\delta$ may of course change, while we can assume without loss of generality that $d$ is the same), the functions $\eta_{1,\rho,\delta}$ replaced by $\eta_{2,\rho,\delta}$ (which are positive on $E_{2,\rho,\delta}$) and the probability measures $\nu_{1,\rho,\delta}$ replaced by the probability measures $\nu_{2,\rho,\delta}$ on $E_{2,\rho,\delta}$, and $E'_1$ replaced by $E'_2:=E_2\setminus \cup_{\rho,\delta} E_{2,\rho,\delta}$.
		
		Repeating this procedure, we build iteratively $E_3,E'_3,\ldots,E_p,E'_p,E'_{p+1}$, for some $p\geq 2$. Note that each $\eta_{i,\rho,\delta}$ is an eigenfunction associated to the eigenvalue $r(P)^d$ for the semigroup associated to $P^d$ restricted to $E\setminus (E_1\cup\ldots\cup E_{i-1})$. One easily checks that, to each such eigenfunction, one can build an element $\eta'_{i,\rho,p}$ of $\ker (P^d-r(P)^d I)^{i}$ on $E$ such that $\eta$ coincides with $\eta_{i,\rho,\delta}$  on $E_{i,\rho,\delta}$. Since $P$ is quasi-compact, $\sum_{i,d}\ker (P^d-r(P)^d I)^{i}$ is finite dimensional in $B$, and hence, since the $\eta'_{i,\rho,p}$ form a free family in $B$, we deduce that 
		\[
		p\leq \dim\left(\sum_{i,d}\ker (P^d-r(P)I)^{i}\right)<\infty.
		\]
		
			\medskip\noindent \textit{Step~3. Convergence at times multiples of $d$.} We are now in position to apply \cite[Theorem~4.1]{ChampagnatVillemonais2024} to the kernel $P^d$. Indeed, the state space $E$ can be decomposed in $E_0=\cup_i E'_i$ and the subsets $E_{i,\rho,\delta}$. Assumption~(B1) therein holds with $W_{i,\rho,\delta}=V_{\vert E_{i,\rho,\delta}}$ according to~\eqref{eq:criticalsink}, Assumption~(B2) holds since communication from  the set $E_{i,\rho,\delta}$ to the set $E_{i',\rho',\delta'}$ is only possible (but not necessarily) if $i<i'$, Assumption (B3) is easily obtained from~\eqref{eq:sink} with $W_0=V_{\vert E_0}$ and, finally, Assumption~(B4) holds since $\sum_i W_i=V$ and $P$ is a bounded operator in $L^\infty(V)$. We deduce from~\cite[Theorem~4.1]{ChampagnatVillemonais2024} that the process $(X_{nd})_{n\geq 0}$ satisfies Assumption~(A) there: 
			there exists an integer $d\geq 1$, an integer valued function $j:E\to \mathbb N$, a finite set $I$ and some probability measures $\nu_{i}$ on $E$ such that $\nu_i(V)<+\infty$ and non-identically zero non-negative functions $\eta_{i}\in L^\infty(V)$ for each $i\in I$, such that, for all $f\in L^\infty(V)$, all $n\geq 1$ and all $x\in
			E$,
			\begin{equation}
				\label{eq:etaproof}
				\left|r(P)^{-nd} (nd)^{-j(x)} \mathbb E_x(f(X_{dn})\mathbf 1_{nd< \tau_\partial})-\sum_{i\in I} \eta_{i}(x)\nu_{i}(f)\right|\leq \alpha_{n} V(x)\|f\|_{V},
			\end{equation}
			where $\alpha_{n}$ goes to $0$ when $n\to+\infty$. In addition, the fact that $r(P^d)=r(P)^d$ and \cite[Theorem~4.1]{ChampagnatVillemonais2024} also entails the last statement of Theorem~\ref{thm:quasi-compact-and-QSD}.
	
		\medskip\noindent\textit{Step 4. Conclusion.} Let $\nu_{i,k}=r(P)^{-k}\nu_i P^k$, so that $\nu_{i,k+d}=\nu_{i,k}$ for all $i,k$. Applying~\eqref{eq:etaproof} to $P^k f$ instead of $f$, we get
		\begin{equation*}
			\left|r(P)^{-nd} (nd)^{-j(x)} \mathbb E_x(f(X_{dn+k})\mathbf 1_{nd+k< \tau_\partial})-\sum_{i\in I} \eta_{i}(x)\nu_{i,k}(f)\right|\leq \alpha_{n} V(x)\|f\|_{V}.
		\end{equation*}
		and hence, setting $\eta_{i,k}=r(P)^{-k}\eta_i$,
		\begin{equation}
			\label{eq:etaproofandk}
			\left|r(P)^{-nd-k} (nd)^{-j(x)} \mathbb E_x(f(X_{dn+k})\mathbf 1_{nd+k< \tau_\partial})-\sum_{i\in I} \eta_{i,k}(x)\nu_{i,k}(f)\right|\leq r(P)^{-d}\alpha_{n} V(x)\|f\|_{V}.
		\end{equation}
		In addition, we have
		\begin{align*}
			r(P)^{-nd-k} & \left|(nd)^{-j(x)}- (nd+k)^{-j(x)} \right|\mathbb E_x(|f(X_{dn+k})|\mathbf 1_{nd+k< \tau_\partial})\\
			&\leq r(P)^{-nd-k}(nd)^{-j(x)} \mathbb E_x(|f(X_{dn+k})|\mathbf 1_{nd+k< \tau_\partial})\left(\frac{(nd+k)^{-j(x)}}{(nd)^{-j(x)}}-1\right)\\
			&\leq \left(\sum_{i\in I} \eta_{i,k}(x)\nu_{i,k}(|f|)+r(P)^{-d}\alpha_n V(x)\|f\|_V\right)\left(\frac{(nd+k)^{-\max j}}{(nd)^{-\max j}}-1\right)\\
			&\leq r(P)^{-d}\left(\sum_{i\in I} \|\eta_i\|_V V(x)\,\nu_{i,k}(V)\|f\|_V+\max_m \alpha_m V(x)\|f\|_V\right)\left(\frac{(nd+k)^{-\max j}}{(nd)^{-\max j}}-1\right),
		\end{align*}
		where $\frac{(nd+k)^{-\max j}}{(nd)^{-\max j}}-1\to 0$ when $n\to+\infty$.  This and~\eqref{eq:etaproofandk} allows us to conclude the proof (up to a change in $\alpha_n$).
\end{proof}

We conclude with two interesting special cases, first  when the kernel is irreducible in the sense of Definition~\ref{def:irr} below, and the other one when the state space is topologically irreducible with  a lower semi-continuous type property.

\begin{definition}
	\label{def:irr}
	We say that the kernel $(P_n)_{n\in\mathbb N}$ is \textit{totally irreducible} if, for all $y\in E$ and $A\in\mathcal E$,
	\begin{align*}
		\exists n\geq 1\,\text{ such that }\,P^n\mathbf 1_A(y)>0 \Longrightarrow \forall x\in E,\ \exists n\geq 1\,\text{ such that }\,P^n\mathbf 1_A(x)>0 .
	\end{align*}
\end{definition}

For a totally irreducible kernel, the set $E_{1,1,1}$ in the proof of Theorem~\ref{thm:quasi-compact-and-QSD} is equal to $E$, which immediately entails the following corollary.

\begin{corollary}
	\label{cor:totirr}
	If, in addition to the assumptions of Theorem~\ref{thm:quasi-compact-and-QSD}, the kernel $P$ is totally irreducible, then  there exists an integer $d\geq 1$, a non-zero measure $\nu$ on $E$ such that $\nu(V)<+\infty$ and a positive  function $\eta\in L^\infty(V)$ such that, for all $f\in L^\infty(V)$, all $n\geq 1$, all $k\in\{0,1,\ldots,d-1\}$ and all $x\in
	E$,
	\begin{equation}
		\label{eq:etacor1}
		\left|r(P)^{-dn-k} P^{n+k}f(x)- \eta(x)\nu_k(f)\right|\leq C\,\rho^{n} V(x)\|f\|_{V},
	\end{equation}
	where $\nu_{k}=\nu P^k$, $C>0$ and $\rho\in (0,1)$. 
\end{corollary}

We focus now on the situation where $E$ is a separable topological space endowed with its Borel $\sigma$-field $\mathcal E$, $P$ is topologically irreducible (see the definition below) and $P$ has some lower semi-continuous property.

\begin{definition}
	\label{def:irrtop}
	We say that the $P$  is \textit{topologically irreducible} if, for all open set $O\subset E$, and all $x\in E$, there exists $n\geq 1$ such that $P^n\mathbf 1_O(x)>0$.
\end{definition}

The following result is then an immediate consequence of Corollary~\ref{cor:totirr}.

\begin{corollary}
	\label{cor:totirrtop}
	If, in addition to the assumptions of Theorem~\ref{thm:quasi-compact-and-QSD}, $E$ is a topological space, $P$ is  topologically irreducible and, for all $A\in\mathcal E$ and all $y\in E$,
	\begin{align*}
	\exists n\geq 1\,\text{ such that }\,P^n\mathbf 1_A(y)>0  \Longrightarrow 	\liminf_{x\to y} \ \sup_{n\geq 1} P^n\mathbf 1_A(x)>0,
	\end{align*}
	then the conclusions of Corollary~\ref{cor:totirr} hold.
\end{corollary}

\subsection{Continuous time semigroups}
\label{sec:QSDbyQCcont}

Let $E$ be a measurable space state space  endowed with a countably generated $\sigma$-field $\mathcal E$, and let  $(P_t)_{t\in[0,+\infty)}$ be a continuous time semigroup such that, for all $t>0$, $P_t$ is a non-negative kernel on $E$.

\begin{theorem}
	\label{thm:generalcont}
	Assume that, for all bounded measurable function $f\geq 0$,  the map $t\in[0,+\infty) \mapsto P_tf(x)$ is measurable and that there exists $T>0$ and a measurable function $V:E\to(0,+\infty)$ such that  $P_T$ is a quasi-compact operator on $L^\infty(V)$. Assume in addition that  $P_t V\leq C_TV$ for all $t \in[0,T]$ and some constant $C_T>0$. Then
	\begin{equation}
		\label{eq:etathmcont}
		\left|\frac{r(P_1)^{-t}}{(1+t)^{j(x)}} P_t f(x)- \sum_{i\in I}\eta_i(x)\nu_i(f)\right|\leq \alpha_t V(x)\|f\|_{V},
	\end{equation}
	where $\alpha_t\to 0$ when $t\to+\infty$, and where the objects $\eta_i,\nu_i,j$ are the same as in the statement of Theorem~\ref{thm:quasi-compact-and-QSD} applied to $P_T$. 
\end{theorem}

\begin{remark}
	\label{rem:2notablefacts}
	There are two notable consequences of the above result on positive continuous time semigroups $(P_t)_{t\geq 0}$ on an $L^\infty(V)$ space. First,  if $P_T$ is quasi-compact for some $T>0$ and $P_t V\leq C_T V$ for all $t\in[0,T]$ and some constant $C_T>0$, then $P_t$ is quasi-compact for all $t> 0$. Second, under theses conditions, $(P_t)_{t\geq 0}$ cannot have a periodic behaviour, contrarily to the discrete time setting, so that, in Theorem~\ref{thm:quasi-compact-and-QSD} applied to $P_T$, one can take $d=1$.
	As far as we know, Thorem~\ref{thm:generalcont} and these two consequences are new also in the context of conservative semigroups.
\end{remark}

\begin{remark}
	Since the above theorem relies on the property of the discrete  time semigroup $(P_{nT})_{n\geq 0}$, all the discrete time setting results (Corollaries~\ref{cor:totirr} and~\ref{cor:totirrtop} and the results of Section~\ref{sec:Lyap}), are useful in the context of continuous time semigroups.
\end{remark}

\begin{remark}
	If $(P_t)_{t\in[0,+\infty)}$ is not measurable in $t$, one can also adapt the discrete time results of the previous sections to the continuous time setting following the classical lines of~\cite{MeynTweedie1993a} or, in the context of quasi-stationary distributions,~\cite{ChampagnatVillemonais2017a}. For instance, 
 the result holds with $I$ of cardinality $1$ if the process satisfies in addition the Doeblin/Lyapunov conditions~\cite[Assumptions~F0, F1, F2]{ChampagnatVillemonais2017a}. The treatment of this situation is very similar to this reference, and thus we do not detail it any further.
	Following this direction thus allows to replace the Harnack type condition F3 in~\cite{ChampagnatVillemonais2017a} with a quasi-compactness condition, which, in some cases and thanks to Theorem~\ref{thm:QCbyDOM}, can be much easier to check.
\end{remark}

\begin{proof}[Proof of Theorem~\ref{thm:generalcont}]
	Assume without loss of generality that $T=1$. According to Theorem~\ref{thm:quasi-compact-and-QSD} applied to the kernel $P=P_1$, we deduce that there exist an integer $d\geq 1$, an integer valued bounded function $j:E\to \mathbb N$, a finite set $I$ and some non-negative measures $\nu_{i}$ on $E$ such that $\nu_i(V)<+\infty$ and non-identically zero non-negative functions $\eta_{i}\in L^\infty(V)$ for each $i\in I$, such that, for all $f\in L^\infty(V)$, all $n\geq 1$ and all $x\in
	E$,
	\begin{equation}
		\label{eq:etathmconttime}
		\left|\frac{r(P_1)^{-dn}}{(dn)^{j(x)}} P_{dn}f(x)-\sum_{i\in I} \eta_{i}(x)\nu_{i}(f)\right|\leq \bar\alpha_{nd} V(x)\|f\|_{V},
	\end{equation}
	where  $\bar\alpha_{n}$ goes to $0$ when $n\to+\infty$.

For any fixed $j\in I$, any $h\geq 0$ and $x\in E_{j}$, we deduce, replacing $f$ by $P_h f$ in the above inequality and using the fact that $\eta_i(x)=0$ for all $i\neq j$,  that
\begin{equation*}
	\left|\frac{r(P_1)^{-dn}}{(dn)^{j(x)}} P_{dn+h}f(x)- \eta_{j}(x)\nu_{j}(P_h f)\right|\leq \bar\alpha_{nd} V(x)\|P_h f\|_{V}.
\end{equation*}
Similarly, integrating~\eqref{eq:etathmconttime} with respect to  $\delta_x P_h$, we obtain
\begin{equation*}
	\left|\frac{r(P_1)^{-dn}}{(dn)^{j(x)}} P_{dn+h}f(x)-\sum_{i}\delta_x P_h \eta_{i}\,\nu_{i}(f)\right|\leq \bar\alpha_{nd} \delta_x P_hV\|f\|_{V}.
\end{equation*}
Since $P_h V(x)<\infty$ by assumption, we deduce from the two previous inequalities that, for all $h\geq 0$, 
\begin{align}
	\label{eq:etajdecomp}
\eta_{j}(x)\,\nu_{j}(P_hf)=\sum_{i\in I} \delta_x P_h \eta_{i}\,\nu_{i}(f).
\end{align}
In particular, there exists a finite non-negative matrix $A_h$ indexed by $I\times I$ such that
\[
r(P_1)^{-h} \nu_{j}P_h =\sum_{i\in I}[A_h]_{ji} \nu_i.
\]
Using the vectorized notation $\nu$ for the (column) vector $(\nu_j)_{j\in I}$, we deduce that
\[
r(P_1)^{-h} \nu P_h= A_h \nu,\text{ so that }r(P_1)^{-s-t} \nu P_sP_t= r(P_1)^{-t} A_s \nu P_t=A_s A_t \nu.
\]
Since $(\nu_j)_{j\in I}$ forms a free family and since $r(P)^{-s-t} \nu P_sP_t=r(P)^{-s-t} \nu P_{s+t}=A_{s+t}\nu$, we observe that $A_t$ is uniquely determined for all $t>0$ and hence that, for all $s,t\geq 0$, 
\[
A_sA_t=A_{s+t}.
\]
We claim that, denoting $\mathtt I$ the identity matrix indexed by $I\times I$, 
\begin{align}
	\label{eq:Aisconstant}
	A_t=\mathtt I, \quad\forall t\geq 0.
\end{align}
Once this is proved, we can conclude: it implies that, for all $h\in[0,d]$, $\nu_j P_h =r(P_1)^h \nu_j$ for all $j\in E$ and hence, from~\eqref{eq:etathmconttime} applied to $P_h f$ instead of $f$, that
\begin{align*}
	\left|\frac{r(P_1)^{-dn}}{(dn)^{j(x)}} P_{dn+h}f(x)-r(P_1)^h\sum_{i\in I} \eta_{i}(x)\nu_{i}(f)\right|&\leq \bar \alpha_{nd} V(x)\|P_h f\|_{V}\\
	&\leq C_T \bar\alpha_{nd} V(x)\| f\|_{V}.
\end{align*}
This entails~\eqref{eq:etathmcont} (under the assumption that $T=1$) with
\[
\alpha_t:=C\,\bar\alpha_{\lfloor t/nd \rfloor nd}+{C}{(1+t)^{-\|j\|_{\infty}-1}}\mathbf 1_{\|j\|_{\infty}\geq 1},
\]
for some constant $C>0$. Without assuming that $T=1$, we thus get after time rescaling
\begin{equation*}
	\left|\frac{r(P_T)^{-t/T}}{(1+t/T)^{j(x)}} P_t f(x)- \sum_{i\in I}\eta_i(x)\nu_i(f)\right|\leq \alpha_{t/T} V(x)\|f\|_{V},
\end{equation*}
This implies that $P_1$ is quasi-compact, and hence we can actually take $T=1$, which proves~\eqref{eq:etathmcont}.

It remains to prove~\eqref{eq:Aisconstant}. Since $r(P)^{-nd}\nu P_{dn}=\nu$ for all $n\geq 0$, we have 
\[
A_{nd}=\mathtt I,\ \forall n\geq 0.
\]
As a consequence, for all $t\geq 0$, $A_tA_{\lfloor t/d\rfloor d+d-t }=\mathtt I$ and hence $A_t$ is invertible with inverse $A_{\lfloor t/d\rfloor d+d-t }$, which is a matrix with non-negative entries. In particular, we deduce from~\cite{DeMarr1972} that, for all $t\geq 0$, $A_t$ is the product of a diagonal matrix $D_t$ and a permutation matrix $Q_t$. Since this decomposition is unique, we deduce from the properties of $(A_t)_{t\geq 0}$ that, for all $s,t\geq 0$,
\[
Q_{s+t}=Q_sQ_t.
\]
Setting $Q_t:=Q_{-t}^{-1}$ for all $t<0$, the family $(Q_t)_{t\in\mathbb R}$ defines a measurable homeomorphism between the topological groups $(\mathbb R,+)$ and $(S_I,\times)$ (where $S_I$ denotes the group of permutation matrices indexed by $I\times I$). Since $\mathbb R$  has the Baire property and $S_I$ is separable, we deduce from Pettis' theorem, and more precisely from its corollary \cite[Theorem~9.10]{Kechris2012}, that $Q$ is continuous in $t\in \mathbb R$. Since the set $S_I$ is discrete, we deduce that $Q$ is constant and hence that
\[
Q_t=Q_0=\mathtt I,\ \forall t\geq 0.
\]
This and the semigroup property for $A$ entails that $D_{s+t}=D_s D_t$ for all $s,t\geq 0$ and since $t\mapsto D_t$ is bounded over finite intervals (since this is clearly the case for $A_t$), we deduce that the diagonal elements of $D_t$ are of the form $e^{c_i t}$ for some $c_i\in\mathbb R$, $i\in \mathtt I$. Since $D_{nd}=\mathtt I$ for all $n\geq 0$, we deduce that $c_i=0$ for all $i\in \mathtt I$ and hence that 
\[
D_t=\mathtt I,\ \forall t\geq 0.
\]
We have thus proved the claim~\eqref{eq:Aisconstant}, which concludes the proof of Theorem~\ref{thm:generalcont}.
\end{proof}

\section{Applications to quasi-stationary distributions}

\label{sec:applis}

This final section applies the preceding operator‑theoretic results to absorbed Markov processes, providing a unified framework for quasi‑stationary distributions in very general settings.  In particular, we recover and extend strong Feller Lyapunov frameworks from the literature, showing that they fit into a more general domination picture and that similar conclusions hold for non‑Feller and dominated processes that were previously out of reach.

Let $(X_t)_{t\in \mathbf I}$, with $\mathbf I=\mathbb N$ or $\mathbf I=[0,+\infty)$, be a discrete time or continuous time absorbed Markov process on a countably generated measurable space $E\cup\{\partial\}$ with $\partial\notin E$ absorbing, meaning that $X_t=\partial$ for all $t\geq \tau_\partial:=\inf \{s\geq 0,\ X_s=\partial\}$ (see~\cite[Chapter~III]{RogersWilliamsVol12000}, Definition~(1.1) and Section~3 for the continuous time setting). Define its semigroup $(P_t)_{t\in \mathbf I}$ by
\[
P_tf(x):=\mathbb E_x(f(X_t)\mathbf 1_{t<\tau_\partial}),\ \forall x\in E,\ f\in \mathcal B_b(E).
\]

The aim of this section is to apply the results of the previous sections to the quasi-stationary behaviour of $X$, that is on providing estimates on the limiting behaviour of $\mathbb P_x(X_t\in\cdot,\,t<\tau_\partial)$ using Theorems~\ref{thm:quasi-compact-and-QSD} and~\ref{thm:generalcont} by identifying situations where $P_T$ is quasi-compact in  $L^\infty(V)$ space for some $V:E\to(0,+\infty)$ and some $T>0$. 

In order to do so, we first observe that, if $\mathbf I=\mathbb N$,
 the property~\eqref{eq:etathm} translates into
	\begin{equation}
		\label{eq:QSDdisctime}
		\left|\frac{r(P)^{-dn-k}}{(dn+k)^{j(x)}} \mathbb E_x(f(X_{dn+k})\mathbf 1_{nd+k< \tau_\partial})-\sum_{i\in I} \eta_{i,k}(x)\nu_{i,k}(f)\right|\leq \alpha_{nd+k} V(x)\|f\|_{V}.
	\end{equation}
	If in addition $V\geq \mathbf 1_E$, then each $\nu_i$, $i\in I$, can be chosen as a probability measure and is a so-called quasi-stationary distribution for $X$, meaning that
	\[
	\mathbb P_{\nu_i}(X_n\in\cdot\mid n<\tau_\partial)=\nu_i,\ \forall n\geq 0.
	\]
		We refer the reader to~\cite{ChampagnatVillemonais2024}, where the authors study quasi-stationary distributions of processes satisfying~\eqref{eq:QSDdisctime}.
		
Similarly, if $\mathbf I=[0,+\infty)$, the property~\eqref{eq:etathmcont}
translates into
\begin{equation}
	\label{eq:QSDconttime}
			\left|\frac{r(P_1)^{-t}}{(1+t)^{j(x)}} \mathbb E_x(f(X_{t})\mathbf 1_{t< \tau_\partial})- \sum_{i\in I}\eta_i(x)\nu_i(f)\right|\leq \alpha_t V(x)\|f\|_{V},
\end{equation}
In both cases, this entails that, for all $x\in E$ such that the limiting term is non zero, $\mathbb P_x(X_t\in \cdot\mid t<\tau_\partial)$ converges when $t\to+\infty$ to a measure. In addition, if $V\geq 1$, then this measure is a probability measure and hence a quasi-stationary distribution (see e.g.~\cite{Meleard2012}).

In Section~\ref{sec:exadis}, we consider discrete time lazy Markov chains with killing. In Section~\ref{sec:exadensity}, we study continuous time processes with locally uniformly integrable densities. In Section~\ref{sec:exaFeller}, we treat the case of continuous time processes with a semigroup dominated by a strongly Feller kernel.

\subsection{Lazy Markov chain with killing}
\label{sec:exadis}

Let $(E,\mathcal E)$ be measurable spaces with $\mathcal E$ countably generated and let $R$ be a Markov transition kernel on $E$, and let $\partial \notin E$. We consider the Markov process $X$ on $E\cup\{\partial\}$ with the following transition dynamics:
\[
\mathbb P_x(X_{1}\in\cdot)=\rho_R(x)R(x,\cdot)+\rho_\delta(x)\,\delta_{x}+\rho_\partial(x)\delta_\partial,
\]
where $\rho_R$, $\rho_\delta$ and $\rho_\partial$ are non-negative measurable functions on $E$ such that $\rho_R+\rho_\delta+\rho_\partial\equiv 1$, and $\rho_\partial(\partial)=1$.
Said differently, this is a simple Markov process which, when at a point $x$, evolves according to the transition $R$ with probability $\rho(x)$, is lazy (i.e. it stays in $x$) with probability $\rho_\delta(x)$, and is sent to the cemetery point $\partial$ otherwise. Note that $\partial$ is absorbing.

\begin{corollary}
	\label{cor:simpleexample}
	If there is a probability measure $\mu$ on $E$ such that 
	\[
	a:=\left\|\frac{\mathrm dR(x,\cdot)}{\mathrm d\mu}\right\|_\infty<\infty
	\]
	and if $\|\rho_\delta\|_\infty+\|\rho_\partial\|_\infty<1$, then~\eqref{eq:QSDdisctime} holds with $V\equiv 1$.
\end{corollary}

This example, despite its apparent simplicity, doesn't fit in existing framework for the existence of and convergence to a quasi-stationary distribution (apart when $E$ is countable). The main reasons are that the family of laws $\mathbb P_x(X_1\in\cdot)$, $x\in E$, is too degenerate to make use of~\cite{ChampagnatVillemonais2017a} (because of the $\delta_x$ is the definition) and too irregular to make use e.g. of~\cite{FerreRoussetEtAl2018,CastroLambEtAl2021,BenaimEtAl2022}  (since we do not make any assumption on the regularity of $R$ nor on the functions $\rho_R,\rho_\delta$ and $\rho_\partial$). 

\begin{proof}[Proof of Corollary~\ref{cor:simpleexample}]
	Let $P=P_1$ be the operator on $L^\infty(E)$ associated with $X$. Then we have
	$
	0\leq P\leq K+S,
	$
	where $Kf(x):=a\mu(f)$ and $Sf(x):=\rho_\delta(x) f(x)$ define positive operators such that $K$ is compact (since its range is one-dimensional) and $\rho(S)\leq\|\rho_\delta\|_\infty<1-\|\rho_\partial\|_\infty\leq r(P)$. We conclude from Theorem~\ref{thm:QCbyDOM} that $P$ is quasi-compact, and by Theorem~\ref{thm:quasi-compact-and-QSD} that ~\eqref{eq:QSDdisctime} holds.
\end{proof}

\subsection{Kernel with locally uniformly integrable densities}

\label{sec:exadensity}

In this section, we consider the situation where a continuous time absorbed process admits a locally uniformly integrable density. This type of properties is well known for many families of continuous time processes since the seminal work of Aronson~\cite{Aronson1968}, under mild regularity assumptions, see e.g. the recent works on kinetic stochastic differential equations~\cite{deRaynalEtAl2023,HouXicheng2024,RenXicheng2025} and references therein. We also refer the reader to~\cite{DebusscheFournier2013,MenozziZhang2022} and references therein for processes driven by $\alpha$-stable processes.

Let $(X_t)_{t\in[0,+\infty)}$ be a continuous time Markov process on a separable topological space $E\cup\{\partial\}$, with $E$ endowed with its Borel $\sigma$-field $\mathcal E$ and $\{\partial\}\not\subset E$ is measurable and absorbing subset. We denote the absorption time of $X$ by $\tau_\partial:=\inf\{t\geq 0,\ X_t=\partial\}$ and assume that, for all $x\in E$ and all $f\in L^\infty(E)$, $t\in[0,+\infty)\mapsto \mathbb E_x(f(X_t)\mathbf 1_{t<\tau_\partial})$ is measurable.

\begin{proposition}
	\label{prop:QSDdensity}
	Assume that there exists a measurable function $V:E\rightarrow (0,+\infty)$,
	$T>0$ and a constant 
	$C_T>0$ such that
	\[
	\mathbb E_x(V(X_t)\mathbf 1_{t<\tau_\partial})\leq C_T V(x),\quad \forall t\in[0,T],\ x\in E.
	\]
	Assume in addition that there exists $\theta_1<\theta'_1$, a non-negative non-zero
	$\varphi\in L^\infty(V)$ and a measurable $E_K\subset E$ such that
	\begin{align*}
		&\mathbb E_x(V(X_T)\mathbf 1_{T<\tau_\partial})\leq \theta_1 V(x),\ \forall x\in E\setminus E_K\\
		&\mathbb E_x(\varphi(X_T))\geq \theta'_1 \varphi(x),\ \forall x\in E.
	\end{align*}	
	If there exists a measure $\nu$ on $E$ and a measurable $p_T:E\times E\to [0,+\infty)$ such that, for all $x\in E_K$,
	\[
	\mathbb E_x(f(X_T)\mathbf 1_{T<\tau_\partial})=\int_E f(y) p_T(x,y)\,\nu(\mathrm dy)
	\]
	with one of the two following conditions:
	\begin{enumerate}[i)]
		\item we have $\theta_1 C_T<\theta'_1$,  $\nu(V\mathbf 1_{E_K})<\infty$ and 
		\[
		\sup_{x\in E_K} \frac1{V(x)}\int_{E_K} p_T(x,y)\mathbf 1_{p_T(x,y)>B}V(y)\,\nu(\mathrm dy)\xrightarrow[B\to+\infty]{} 0;
		\]
		\item for all $A>0$, $\nu(V\mathbf 1_{V\leq A})<\infty$ and
		\[
		\sup_{x\in E_K}  \frac1{V(x)}\int_{V\leq A} p_T(x,y)\mathbf 1_{p_T(x,y)>B}V(y)\,\nu(\mathrm dy)\xrightarrow[B\to+\infty]{} 0;
		\]
	\end{enumerate}
	then~\eqref{eq:QSDconttime} holds.

	If in addition, for all $x\in E_K$, $p_T(x,y)$ is positive $\nu(\mathrm dy\cap E_K)$-almost everywhere, then the cardinality of $I$ in~\eqref{eq:QSDconttime} is one and $\eta_1$ is positive on $E_K$. 
\end{proposition}

\begin{remark}
	As will be clear in the proof, the condition that there exists $\varphi$ and $\theta'_1$ such that $\mathbb E_x(\varphi(X_T))\geq \theta'_1 \varphi(x)$ for all $x\in E$ can be removed if one replaces $\theta'_1$ by $r(P_T)$ in the above statement. Similarly, $\theta_1 C_T<\theta'_1$ can be replaced by $\theta_1 \VERT P_T\VERT < r(P_T)$.
\end{remark}

\begin{proof}[Proof of Proposition~\ref{prop:QSDdensity}]
	Let $P_T$ be the kernel associated to $X_T$, which defines a bounded linear operator on $L^\infty(V)$ with $C_T\geq \VERT P_T\VERT $ since $P_T V\leq C_T V$.
	Following Remark~\ref{rem:phi}, we also have $r(P_T)\geq \theta'_1$. 
	
	We deduce that $P_T$ satisfies Assumption~(H1) and Conditions {\em i)} or {\em ii} in Proposition~\ref{prop:density}. In particular $P_T$ is quasi-compact and we deduce from Theorem~\ref{thm:generalcont} that~\eqref{eq:QSDconttime} holds.
	
	Assume in addition that, for all $x\in E_K$, $p_T(x,y)$ is positive $\nu(\mathrm dy\cap E_K)$-almost everywhere. According to Theorem~\ref{thm:quasi-compact-and-QSD} and Remark~\ref{rem:2notablefacts}, each $E_i$, $i\in I$, is closed for $P_T$, so that there exists $i_0\in I$ such $E_K\subset E_{i_0}$. In particular, for all $i\in I\setminus\{i_0\}$,
	\[
	\mathbb E_x(V(X_T)\mathbf 1_{T<\tau_\partial})\leq \theta_1 V(x),\quad\forall x\in E_i.
	\]
	Integrating with respect to $\nu_i$, and using also the fact that $\nu_iP_T=r(P_T)\nu_i$, we deduce that 
	\[
	r(P_T)\nu_i(V)=\mathbb E_{\nu_i}(V(X_T)\mathbf 1_{T<\tau_\partial})\leq \theta_1 \nu_i(V),\quad\forall x\in E_i,
	\]
	and hence that $r(P_T)\leq \theta_1$. Since $\theta_1<\theta'_1\leq r(P_T)$, we deduce that $I\setminus\{i_0\}$ is empty, which concludes the proof.
\end{proof}

\subsection{Sub-Markov processes dominated by a strong Feller kernel}

\label{sec:exaFeller}

In this section, we consider the situation where a continuous time absorbed process is dominated by a strong Feller semigroup. 

Let $E$ be a separable metric space endowed with its Borel $\sigma$-field and let $(X_t)_{t\in [0,+\infty)}$ be as in the previous section. We have the following result, whose proof is similar to the proof of Proposition~\ref{prop:QSDdensity}, using Corollary~\ref{cor:loc1} instead of Corollary~\ref{cor:localdomqsdbis}, and is thus omitted.

\begin{proposition}
	\label{prop:strongFeller}
		Assume that there exists a measurable function $V:E\rightarrow (0,+\infty)$,
	$T>0$ and a constant 
	$C_T>0$ such that
	\[
	\mathbb E_x(V(X_t)\mathbf 1_{t<\tau_\partial})\leq C_T V(x),\quad \forall t\in[0,T],\ x\in E.
	\]
	Assume in addition that there exists $\theta_1<\theta'_1$, a non-negative non-zero
	$\varphi\in L^\infty(V)$ and a measurable $E_K\subset E$ such that
	\begin{align*}
		&\mathbb E_x(V(X_T)\mathbf 1_{T<\tau_\partial})\leq \theta_1 V(x),\ \forall x\in E\setminus E_K\\
		&\mathbb E_x(\varphi(X_T))\geq \theta'_1 \varphi(x),\ \forall x\in E.
	\end{align*}
	Assume in addition one of the following
	\begin{enumerate}[i)]
		\item $E_K$ is relatively compact, and there exists $\theta_3>0$ and a strong Feller kernel $G$ on $E_K$ such that
		\begin{align*}
			\mathbb E_x(f(X_T)\mathbf 1_{X_T\in E_K}\mathbf 1_{T<\tau_\partial})\leq G(f\mathbf 1_{E_K})(x)+\theta_3 \|f\|_V\,V(x),\ \forall x\in E_K,\ \forall f\in L^\infty(V),
		\end{align*}
		and with $(\theta_3+\theta_1)(C_T^2+\VERT G\VERT \, C_T)<(\theta'_1)^3$;
		\item for all $A>0$ $\{V\leq A\}$ is relatively compact, and we have $\sup_{E_K} P(V\mathbf 1_{V>A})\to 0$ when $A\to+\infty$, and there exists a  strong Feller kernel $G$ on $E$ such that, for some $\theta_4<\theta'_1$,
		\begin{align*}
			\mathbb E_x(f(X_T)\mathbf 1_{T<\tau_\partial})\leq Gf(x)+\theta_4 \|f\|_V\,V(x),\ \forall x\in E_K,\ \forall f\in L^\infty(V);
		\end{align*}
	\end{enumerate}
	Then~\eqref{eq:QSDconttime} holds.
\end{proposition}

\begin{remark}
As in the conclusion of Proposition~\ref{prop:QSDdensity}, if irreducibility conditions are added to the assumption of this result, then one can conclude that the cardinality of $I$ in~\eqref{eq:QSDconttime} is one and $\eta_1$ is positive on $E_K$. 
\end{remark}

\begin{remark}
In~\cite{FerreRoussetEtAl2018}, the authors show that discrete time strong Feller non-con\-servative processes satisfying Lyapunov type conditions entail that the operator is quasi-compact. More recently, in~\cite{GuillinNectouxEtAl2024}, the authors proved using a different approach that this also holds for continuous time strong Feller sub-Markov processes, under the condition that they are obtained from càdlàg strong Feller Markov processes killed at the boundary of an open set. Proposition~\ref{prop:QSDdensity} {\em ii)} with $\theta_4=0$ generalizes this result by showing that it is actually sufficient that the sub-Markov process is dominated by a strong Feller kernel to obtain quasi-compactness.
\end{remark}

\end{document}